\documentclass[a4paper]{amsart}

\usepackage{amsmath,amssymb,amsthm}
\usepackage[all]{xy}

\DeclareMathOperator{\Diff}{Diff}
\DeclareMathOperator{\Aut}{Aut}

\DeclareMathOperator{\GL}{GL}
\DeclareMathOperator{\SL}{SL}
\DeclareMathOperator{\hol}{hol}

\DeclareMathOperator{\tr}{tr}

\newcommand{\OO}{\operatorname{O}}

\newcommand{\ZZ}{{\mathbb  Z}}
\newcommand{\RR}{{\mathbb  R}}

\newcommand{\TT}{{\mathbb  T}}

\newcommand{\F}{{\mathcal F}}

\newtheorem{theorem}{Theorem}[section]
\newtheorem*{theorem*}{Theorem}
\newtheorem{corollary}[theorem]{Corollary}
\newtheorem{lemma}[theorem]{Lemma}
\newtheorem{proposition}[theorem]{Proposition}
\theoremstyle{definition}
\newtheorem{definition}[theorem]{Definition}
\newtheorem{example}[theorem]{Example}
\newtheorem*{question}{Question}

\theoremstyle{remark}
\newtheorem{remark}[theorem]{Remark}

\newtheorem*{org}{Organization of the article}

\begin{document}
\title{Tenseness of Riemannian flows}

\author{Hiraku Nozawa}
\address{Hiraku Nozawa, Department of Mathematical Sciences, Faculty of Science and Engineering, Ritsumeikan University, Nozihigashi 1-1-1, Kusatsu, Shiga, 525-8577, Japan}
\email{hnozawa@fc.ritsumei.ac.jp}
\thanks{The first author is supported by the EPDI/JSPS/IH\'{E}S Fellowship and the Spanish MICINN grant MTM2011-25656. This paper was written during the stay of the first author at Centre de Recerca Matem\`{a}tica (Bellaterra, Spain), Institut Mittag-Leffler (Djursholm, Sweden) and Institut des Hautes \'{E}tudes Scientifiques (Bures-sur-Yvette, France); he is very grateful for their hospitality.}
\author{Jos\'{e} Ignacio Royo Prieto}
\address{Jos\'{e} Ignacio Royo Prieto, Departamento de Matem\'{a}tica Aplicada, Universidad del Pa\'{i}s Vasco UPV/EHU, Alameda de Urquijo s/n 48013 Bilbao, Spain}
\email{joseignacio.royo@ehu.es}
\thanks{The second author is partially supported by the UPV/EHU grants EHU09/04 and EHU12/05 and by the Spanish MICINN grant MTM2010-15471. This paper was written during the stay of the second author at Centre de Recerca Matem\`{a}tica (Bellaterra, Spain); he is very grateful for their hospitality.}
\thanks{The authors wish to thank the referee for the useful remarks and suggestions in order to improve this paper.}

\begin{abstract}
We show that any transversally complete Riemannian foliation $\F$ of dimension one on any possibly non-compact manifold $M$ is tense; namely, $(M,\mathcal{F})$ admits a Riemannian metric such that the mean curvature form of $\F$ is basic. This is a partial generalization of a result of Dom\'inguez, which says that any Riemannian foliation on any compact manifold is tense. Our proof is based on some results of Molino and Sergiescu, and it is simpler than the original proof by Dom\'inguez. As an application, we generalize some well known results including Masa's characterization of tautness.
\end{abstract}
\maketitle

\section{Introduction}

\subsection{Background}
A foliated manifold $(M,\mathcal{F})$ is called {\em taut} if $M$ admits a metric $g$ such that every leaf of $\mathcal{F}$ is a minimal submanifold of $(M,g)$; in other words, $M$ admits a metric such that the mean curvature form of $\F$ is trivial. The tautness of foliated manifolds has been studied from the dynamical or geometric point of view after the characterization of tautness in terms of foliation cycles due to Sullivan~\cite{Sullivan}. For Riemannian foliations, tautness is remarkably of topological nature, and its relation to cohomology has been studied by many authors~\cite{Kamber Tondeur,Carriere,Ghys,Haefliger 2,Molino Sergiescu,El Kacimi Hector,Masa,Alvarez Lopez}. In particular, as conjectured by Carri\`ere~\cite{Carriere} and finally proved by Masa~\cite{Masa} using Sarkaria's smoothing operator~\cite{Sarkaria}, an oriented and transversally oriented Riemannian foliation is taut if and only if the top degree component of the basic cohomology is nontrivial. \'{A}lvarez L\'{o}pez~\cite{Alvarez Lopez} defined the so-called {\em \'{A}lvarez class} to characterize tautness of Riemannian foliations and removed the assumption of the orientability from Masa's characterization.

Based on these works on tautness of Riemannian foliations, Dom\'inguez proved the following result.
\begin{theorem}[{\cite[Tenseness Theorem in p.~1239]{Dominguez}}]\label{Theorem : Tenseness}
Any Riemannian foliation on a closed manifold is tense.
\end{theorem}
Here, recall that a foliated manifold $(M,\mathcal{F})$ is called {\em tense} if $M$ admits a metric such that the mean curvature form of $\F$ is basic. Theorem~\ref{Theorem : Tenseness} can be regarded as a generalization of Masa's characterization of tautness, and has many applications in the study of geometrical and cohomological properties of Riemannian foliations (see, for example,~\cite{Kamber Tondeur 2, Kamber Tondeur 3,Tondeur killing,royo,royo-saralegi-wolak1,royo-saralegi-wolak2}).

\subsection{Main result}\label{sec:dic}
In this article we will generalize Theorem~\ref{Theorem : Tenseness} to Riemannian foliations of dimension one on possibly non-compact manifolds.

There is one remarkable difference between the non-compact and the compact cases. By~\cite[Eq.~4.4]{Kamber Tondeur 2}, the mean curvature form $\kappa$ of a Riemannian foliation with a tense metric on a compact manifold is always closed. On the other hand, there exists a Riemannian foliation on a non-compact manifold with a tense metric whose $\kappa$ is not closed~\cite[Example~2.4]{Cairns Escobales}. Based on this fact, we say that a foliated manifold $(M,\F)$ is {\em strongly tense} if $M$ admits a Riemannian metric such that the mean curvature form of $\F$ is basic and closed.

The main result of this paper is the following.
\begin{theorem}\label{Theorem : Strong tenseness}
Any transversally complete Riemannian foliation of dimension one is strongly tense.
\end{theorem}
We refer to Definition~\ref{definition:TC} for the definition of transversally completeness of Riemannian foliations. Note that some authors use this terminology with a different meaning.

\begin{remark}
For Riemannian foliations of any dimension which can be suitably embedded into a singular Riemannian foliation on a compact manifold, strongly tenseness was proved in~\cite{royo-saralegi-wolak1,royo-saralegi-wolak2} by the application of Dom\'inguez's theorem. For any Riemannian foliation such that the space of leaf closures is compact, the strongly tenseness was proved in~\cite[Theorem~1.9]{Nozawa 2}.
\end{remark}

Based on Theorem~\ref{Theorem : Strong tenseness} we ask the following question.

\begin{question}
Is any complete Riemannian foliation strongly tense?
\end{question}

Here a Riemannian foliation is called complete if the holonomy pseudogroup is complete as a pseudogroup. Due to Salem~\cite[Appendix by Salem]{Molino}, it is the largest known class of Riemannian foliations for which Molino's structure theorems hold.

The first essential point in the proof of Theorem~\ref{Theorem : Strong tenseness} is the following dichotomy, which is specific for dimension one (see Remark~\ref{rem:DicMol} for a preceding result of Molino).
\begin{theorem}\label{Theorem: dichotomy}
Let $M$ be a connected manifold with a transversally complete Riemannian foliation $\mathcal{F}$ of dimension one. Then, one of the following holds:
\begin{enumerate}
\item $(M,\mathcal{F})$ is an $\mathbb{R}$-bundle or
\item the closure of every leaf of $\mathcal{F}$ is compact.
\end{enumerate}
\end{theorem}

Since Theorem~\ref{Theorem : Strong tenseness} is clearly true for $\mathbb{R}$-bundles, it is essential to prove Theorem~\ref{Theorem : Strong tenseness} in the case where the closure of every leaf is compact. In turn, Molino's structure theorem remains true in this case even if $M$ is non-compact (Theorem~\ref{Theorem: Molino}). Thus we can apply some results of Molino and Sergiescu involving reductions of the structure group of torus bundles (Section~\ref{section:reduction}) to show Theorem~\ref{Theorem : Strong tenseness}. Even in the case where $M$ is compact, our proof is new and simpler than the original proof of Theorem~\ref{Theorem : Tenseness} due to Dom\'inguez, as we make no use of Sarkaria's smoothing operator~\cite{Sarkaria}.

\subsection{The \'{A}lvarez class}
For a Riemannian foliation $\F$ on a closed manifold $M$ with a bundle-like metric $g$, the orthogonal projection $\kappa_b$ of the mean curvature form $\kappa$ to the space of basic $1$-forms with respect to the natural inner product is closed~\cite[Corollary~3.5]{Alvarez Lopez}. The cohomology class $[\kappa_b]\in H^1(M/\F)$ is independent of $g$ (\cite[Theorem 5.2]{Alvarez Lopez}) and called the {\em \'{A}lvarez class} of $(M,\F)$. The triviality of the \'{A}lvarez class of $(M,\F)$ characterizes tautness~\cite[Theorem~6.4]{Alvarez Lopez}. In turn, the example~\cite[Example~2.4]{Cairns Escobales} of a Riemannian foliation on a non-compact manifold with basic but non-closed $\kappa$ shows that the \'{A}lvarez class is not defined in general for Riemannian foliations on non-compact manifolds. Nevertheless, Theorem~\ref{Theorem: dichotomy} implies the following result for Riemannian foliations of dimension one (see Section~\ref{sec:mol} for the proof).
\begin{theorem}\label{thm:alv}
Let $(M,\F)$ be a connected manifold with a transversally complete Riemannian foliation of dimension one with a strongly tense metric $g$. If $(M,\F)$ is not an $\RR$-bundle, then the cohomology class of the mean curvature form $\kappa$ is given by the logarithm of the holonomy homomorphism $\pi_{1}M \to \RR$ of the determinant line bundle of the Molino's commuting sheaf of $(M,\F)$. In particular, $[\kappa]$ is independent of $g$.
\end{theorem}

We will use the following terminology below, which is well-defined by Theorem~\ref{thm:alv}.
\begin{definition}
For a connected manifold $M$ with a transversally complete Riemannian foliation $\F$ of dimension one which is not an $\RR$-bundle, the cohomology class $[\kappa]$ of the mean curvature form of any strongly tense metric is called the {\em \'{A}lvarez class} of $(M,\F)$.
\end{definition}

\begin{remark}
It is easy to see that an $\RR$-bundle is always taut. But there exists an $\RR$-bundle with a strongly tense metric such that the cohomology class of the mean curvature form is nontrivial (see~\cite[Proposition~9.3]{Nozawa 2}). Below the \'{A}lvarez class of an $\RR$-bundle is defined as the trivial class for a conventional reason.
\end{remark}

\begin{remark}
For Riemannian foliations of any dimension which can be suitably embedded into a singular Riemannian foliation on a compact manifold, the \'Alvarez class is well-defined as shown in~\cite{royo-saralegi-wolak1,royo-saralegi-wolak2} by the application of Dom\'inguez's theorem.
\end{remark}

\subsection{Applications}

\subsubsection{Characterization of tautness}

The basic cohomology is the de Rham cohomology of the leaf space in a sense (see, for example,~\cite[Appendix B]{Molino}), and its relation to tautness of Riemannian foliations was studied by many authors mentioned in the introduction.

First, we state the twisted Poincar\'{e} duality of the basic cohomology, which is a consequence of a theorem of Sergiescu and the dichotomy theorem (Theorem~\ref{Theorem: dichotomy}). In~\cite[Theorem~3.1]{Kamber Tondeur 3}, Kamber-Tondeur proved that any orientable and transversally oriented tense Riemannian foliation $\F$ of codimension $q$ on a compact manifold $M$ satisfies
\begin{equation}\label{eq:tpdKT}
H^{\bullet}_{c}(M/\F)\cong H_{\kappa}^{q-\bullet}(M/\F)^{*}\;,
\end{equation}
where the $\kappa$-twisted basic cohomology $H^{\bullet}_{\kappa}(M/\F)$ stands for the cohomology of the basic de Rham complex with the twisted differential $d_{\kappa}\omega=d\omega - \kappa\wedge\omega$~\cite[p.~121]{Kamber Tondeur 2}. In~\cite[Section~1]{Sergiescu}, Sergiescu defined the orientation sheaf $\mathcal{P}$ of $(M,\F)$ and proved the Poincar\'{e} duality~\cite[Th\'{e}or\`{e}me~I]{Sergiescu} on basic cohomology of $(M,\F)$. His argument shows the isomorphism
\begin{equation}\label{eq:tpdS}
H^{\bullet}_{c}(M/\F)\cong H^{q-\bullet}(M/\F;\mathcal{P}^{*})^{*}
\end{equation}
for any complete Riemannian foliation on a possibly non-compact manifold whose closures of leaves are compact (see~\cite[Proposition~3.2.9.1]{Haefliger 2}). Here a Riemannian foliation is called complete if the canonical transverse parallelism of its lift to the orthonormal frame bundle consists of complete vector fields~\cite[Remark on p.~88]{Molino}. In the case where $M$ is compact, $H_{\kappa}^{\bullet}(M/\F)\cong H^{\bullet}(M/\F;\mathcal{P}^{*})$ by~\cite[Theorem~5.9 (iii)]{Dominguez}. So~\eqref{eq:tpdS} coincides with~\eqref{eq:tpdKT} if $M$ is compact. The twisted duality~\eqref{eq:tpdKT} for Riemannian foliations of any dimension which can be suitably embedded into a singular Riemannian foliation on a compact manifold was proved in~\cite{royo-saralegi-wolak2}. Note that, it is not clear if~\eqref{eq:tpdKT} always follows from~\eqref{eq:tpdS} in the case where $M$ is non-compact.

A priori, a transversally complete Riemannian foliation may not be complete in the sense of~\cite[Remark on page 88]{Molino}, but we get the following result from Theorems~\ref{Theorem: dichotomy},~\ref{thm:alv} and~\cite[Th\'{e}or\`{e}me~I]{Sergiescu}.

\begin{corollary}\label{corollary:duality}
For a transversally oriented and transversally complete Riemannian foliation $\F$ of dimension one and codimension $q$ on a possibly non-compact manifold $M$, we have the isomorphisms~\eqref{eq:tpdS} and
\begin{equation}\label{eq:tpdKT2}
H^{\bullet}_{c}(M/\F)\cong H_{\kappa}^{q-\bullet}(M/\F)^{*}\;,
\end{equation}
where $\kappa$ is a representative of the \'{A}lvarez class of $(M,\F)$.
\end{corollary}

\begin{proof}
If $(M,\F)$ is an $\RR$-bundle, then $\kappa$ is trivial and~\eqref{eq:tpdKT2} follows from the Poincar\'{e} duality of the leaf space $M/\F$. So, by Theorem~\ref{Theorem: dichotomy}, we can assume that the closure of each leaf of $(M,\F)$ is compact. In this case, Molino's structure theorems remain valid (see Theorem~\ref{Theorem: Molino}).
Thus the proof of~\cite[Th\'{e}or\`{e}me~I]{Sergiescu} for the compact case can be applied to show~\eqref{eq:tpdS}. Note that the holonomy homomorphisms of Sergiescu's orientation sheaf $\mathcal{P}$ and the determinant line bundle of Molino's commuting sheaf are equal up to sign by definition of $\mathcal{P}$. Thus the latter part of Theorem~\ref{thm:alv} implies $H_{\kappa}^{\bullet}(M/\F)\cong H^{\bullet}(M/\F;\mathcal{P}^{*})$. Hence we get~\eqref{eq:tpdKT2}.
\end{proof}

Theorem~\ref{Theorem : Strong tenseness} and Corollary~\ref{corollary:duality} give us the following characterization of tautness in terms of basic cohomology (see Section~\ref{sec:mol} for the proof).

\begin{corollary}\label{corollary:tautness characterization}
Let $\F$ be a transversally oriented and transversally complete Riemannian foliation of dimension one and codimension $q$ on a possibly non-compact manifold $M$. Let $\kappa$ be a representative of the \'{A}lvarez class of $(M,\F)$. Then, the following are equivalent:
\begin{enumerate}
\item $\F$ is taut;
\item $H^{q}_c(M/\F) \cong \mathbb{R}$;
\item $H^0_{\kappa}(M/\F) \cong \mathbb{R}$;
\item the image of the holonomy homomorphism $\pi_{1}M \to \Aut(\RR) \cong \RR^{\times}$ of Sergiescu's orientation sheaf of $(M,\F)$ is contained in $\{\pm 1\}$.
\item the image of the holonomy homomorphism $\pi_{1}M \to \Aut(\RR) \cong \RR^{\times}$ of Molino's commuting sheaf of $(M,\F)$ is contained in $\{\pm 1\}$.
\end{enumerate}
Otherwise, $H^{q}_c(M/\F)=0$.
\end{corollary}

Corollary~\ref{corollary:tautness characterization} generalizes~\cite[Th\'{e}or\`{e}me~A]{Molino Sergiescu} and~\cite[Minimality Theorem]{Masa} to Riemannian foliations of dimension one on possibly non-compact manifolds. For Riemannian foliations of any dimension which can be suitably embedded into a singular Riemannian foliation on a compact manifold, the equivalence of the first three statements of Corollary~\ref{corollary:tautness characterization} are shown in~\cite{royo-saralegi-wolak1,royo-saralegi-wolak2} by the application of Dom\'inguez's theorem.

\subsubsection{The Euler class and the Gysin sequence}
In~\cite{royo}, the Euler class and the Gysin sequence of Riemannian flows on compact manifolds were obtained by using Dom\'{i}nguez's tenseness theorem. Theorems~\ref{Theorem : Strong tenseness} and~\ref{Theorem: dichotomy} allow us to obtain the Euler class and the Gysin sequence of transversally complete Riemannian flows on possibly non-compact manifolds.

\begin{corollary}\label{cor:Gysin}
Let $\F$ be an oriented transversally complete Riemannian flow on a possibly non-compact manifold $M$. Then, we get the following long exact sequence:
\begin{equation*}
\xymatrix{ \cdots\rightarrow
H^i(M/\F)\ar[r]
&H^i(M)\ar[r]
&H^{i-1}_{\kappa}(M/\F)\ar[r]^(0.4){\wedge e}
&H^{i+1}(M/\F)\rightarrow\cdots\;,}
\end{equation*}
where $\kappa$ is a representative of the \'{A}lvarez class of $(M,\F)$ and the connecting morphism is the multiplication by the Euler class $e$ defined in the $(-\kappa)$-twisted basic cohomology $H^2_{-\kappa}(M/\F)$.
\end{corollary}

\begin{proof}[Outline of the proof]
If $(M,\F)$ is an $\RR$-bundle, then the Euler class and the Gysin sequence of $(M,\F)$ are trivial. Thus, by Theorem~\ref{Theorem: dichotomy}, it is essential to construct them in the case where the closure of every leaf of $\F$ is compact. In this case, with Theorem~\ref{Theorem : Strong tenseness}, the construction of~\cite{royo} of the Euler class and the Gysin sequence can be carried out without any modification. Note that, since the closure of every leaf is compact, any leaf has a good saturated neighborhood described by Carri\`ere in~\cite[Proposition~3]{Carriere}, which is called a Carri\`ere neighborhood in~\cite{royo}.
\end{proof}

\subsubsection{Generalization of Tondeur's theorem}
By Theorem~\ref{Theorem : Strong tenseness}, we obtain the generalization of the main result of~\cite{Tondeur killing}.

\begin{corollary}
Let $\F$ be a foliation of dimension one on a possibly non-compact manifold $M$. Then $\F$ is transversally complete Riemannian if and only if there exists a complete metric $g$ on $M$ such that the tangent bundle of $\F$ is locally generated by Killing vector fields on $(M,g)$.
\end{corollary}

\begin{proof}[Outline of the proof]
The ``if'' part is proved by an argument similar to~\cite{Tondeur killing} with Theorem~\ref{Theorem : Strong tenseness}. To prove the ``only if'' part, note that, by Theorem~\ref{Theorem: dichotomy}, any transversally complete Riemannian foliation of dimension one admits a complete bundle-like metric.
\end{proof}

\begin{org}
Section~\ref{sec:2}  is devoted to recall the definition of fundamental notions. In Section~\ref{sec:3}, the dichotomy result (Theorem~\ref{Theorem: dichotomy}) is proved. In Section~\ref{sec:4}, we analyze the special case of linearly foliated torus bundles. In Section~\ref{sec:5}, the main result (Theorem~\ref{Theorem : Strong tenseness}) is proved based on Theorem~\ref{Theorem: dichotomy} and the results in Section~\ref{sec:4}.
\end{org}

\section{Fundamental notions}\label{sec:2}

\subsection{Foliations and metrics}
We recall some notions on foliated manifolds and metrics on them. A {\em Haefliger cocycle} (of codimension $q$) on a manifold $M$ is a triple $(\{U_i\},\{\pi_i\},\{\gamma_{ij}\})$ consisting of
\begin{enumerate}
\item an open covering $\{U_i\}$ of $M$,
\item submersions $\pi_i:U_i\to \RR^q$,
\item local diffeomorphisms $\gamma_{ij}:\pi_j(U_i\cap U_j)\to \pi_i(U_i\cap U_j)$ such that $\pi_i=\gamma_{ij}\circ \pi_j$.
\end{enumerate}
Two Haefliger cocycles on $M$ are said to be equivalent if their union becomes a Haefliger cocycle on $M$ after considering the necessary additional maps $\gamma_{ij}$. Recall that a codimension $q$ foliation of $M$ is defined by an equivalence class of Haefliger cocycles of codimension $q$.

A foliation $\F$ is called {\em Riemannian} if there exist Riemannian metrics $h_{i}$ on $\pi_{i}(U_{i})$ such that $\gamma_{ij}^{*}h_{i} = h_{j}$. Let $\nu \F$ denote the normal bundle $TM/T\F$ of $(M,\F)$. Here $(\pi_{i})_{*} : \nu_{x}\F \to T_{\pi_{i}(x)}\RR^{q}$ is an isomorphism at each point $x\in U_{i}$. By pulling back the metric $h_{i}$ by $(\pi_{i})_{*}$ to $\nu_{x}\F$ at each point $x\in U_{i}$, we get a metric on $\nu \F|_{U_{i}}$. This gives rise to a well-defined metric $g$ on $\nu \F$. Such metric on $\nu \F$ constructed from $\{h_{i}\}$ is called {\em holonomy invariant}. We will say that a metric $g$ on $(M,\F)$ is {\em bundle-like} if the metric induced on $\nu \F$ via the identification $\nu \F \cong (T\F)^{\perp}$ is holonomy invariant.

It is easy to see that any manifold with a Riemannian foliation admits a bundle-like metric. In~\cite[Proposition~{2}]{Reinhart} (see also~\cite[Proposition~3.5]{Molino}), Reinhart proved that a metric on $(M,\F)$ is bundle-like if and only if a geodesic whose initial vector is orthogonal to $\F$ is orthogonal to $\F$ everywhere. The following is the notion of completeness in the transverse direction of Riemannian foliations.

\begin{definition}\label{definition:TC}
A Riemannian foliation $\F$ on a connected manifold $M$ is {\em transversally complete} if there exists a Riemannian metric $g$ which is both bundle-like and transversally complete; namely, at every point, on any maximal geodesic that is orthogonal to the leaves, the natural parameter changes from $-\infty$ to $\infty$.
\end{definition}

\begin{remark}
The definition of transversally completeness we use here is different from the one adopted in~\cite[Definition~4.1]{Molino}.
\end{remark}

\subsection{Tautness and tenseness}
For a given Riemannian metric $g$ on a foliated manifold $(M,\F)$, the {\em mean curvature form} $\kappa\in\Omega^1(M)$ of $\F$ at $x$ is the mean curvature form of the leaf which goes through $x$ (see, for example,~\cite[Section~10.5]{Candel Conlon} for formulas of $\kappa$ in terms of $g$).

Recall that, on a foliated manifold $(M,\F)$ represented by a Haefliger cocycle $(\{U_i\},\{\pi_i\},\{\gamma_{ij}\})$, a $k$-form $\alpha$ is called {\em basic} if, for every $i$, there exists a $k$-form $\alpha_{i}$ on $\pi_{i}(U_{i})$ such that $\alpha|_{U_{i}} = \pi_{i}^{*}\alpha_{i}$. Let us recall the following terminologies.

\begin{definition}
A Riemannian metric $g$ on a foliated manifold $(M,\F)$ is said to be {\em tense} (resp., {\em taut}) if the mean curvature form $\kappa$ is basic (resp., trivial).
\end{definition}

We similarly define the following notion according to strongly tenseness of foliated manifolds.
\begin{definition}
A Riemannian metric $g$ on a foliated manifold $(M,\F)$ is said to be {\em strongly tense} if the mean curvature form $\kappa$ is basic and closed.
\end{definition}
Strongly tenseness can be considered as a variant of tautness twisted with a real line bundle (see~\cite[Proposition~7.5]{Nozawa 2}).

\begin{remark}
By a result of Kamber-Tondeur~\cite[Eq.~4.4]{Kamber Tondeur 2}, if $M$ is compact, then any tense metric on $(M,\F)$ is strongly tense.
\end{remark}

\subsection{Characteristic forms}
For a Riemannian manifold $(M,g)$ with an oriented $p$-dimensional foliation $\F$, the characteristic form $\chi\in\Omega^p(M)$ is defined by
\begin{equation}\label{formula:def chii}
\chi(X_1,\dots,X_p)=\operatorname{det}(g(X_i,E_j)_{ij})\;,\quad\forall X_1,\dots,X_p\in C^{\infty}(TM)\;,
\end{equation}
where $\{E_1,\dots,E_p\}$ is a local oriented orthonormal frame of $T\F$.

The mean curvature form is determined by the characteristic form by Rummler's formula~\cite{Rummler}:
\begin{equation}\label{formula:rummler}
\kappa(Y)=-d\chi(Y,E_1,\dots,E_p)\;,\quad \forall Y\in C^{\infty}((T\F)^{\perp g})\;,
\end{equation}
where $\{E_1,\dots,E_p\}$ is a local oriented orthonormal frame of $T\F$. Notice that $\chi$ is determined by the orthogonal complement $(T\F)^{\perp g}$ and the metric along the leaves, hence so is $\kappa$.

\begin{definition}
The characteristic form $\chi$ of an oriented foliation $\F$ on a Riemannian manifold $(M,g)$ is said to be {\em tense}, {\em strongly tense} or {\em taut} if $g$ is tense, strongly tense or taut, respectively.
\end{definition}

\section{A dichotomy on leaves}\label{sec:3}

Let $M$ be a smooth manifold with a transversally complete Riemannian foliation $\mathcal{F}$ of dimension one. We show the following dichotomy.
\setcounter{section}{1}
\setcounter{theorem}{3}
\begin{theorem}[{\bf bis}.]
Let $M$ be a connected manifold with a transversally complete Riemannian foliation $\mathcal{F}$ of dimension one. Then, one of the following holds:
\begin{enumerate}
\item $(M,\mathcal{F})$ is an $\mathbb{R}$-bundle or
\item the closure of every leaf of $\mathcal{F}$ is compact.
\end{enumerate}
\end{theorem}
\setcounter{section}{3}
\setcounter{theorem}{1}

\begin{remark}\label{rem:DicMol}
A result~\cite[Lemme~3]{Molino 2} of Molino implies that the leaves of any Lie foliation of dimension one which is transversally complete in the sense of~\cite[Definition 4.1]{Molino} are either all closed or all have compact closures. Theorem~\ref{Theorem: dichotomy} is its generalization proved by a similar argument.
\end{remark}

We will need the following lemma to prove Theorem~\ref{Theorem: dichotomy}.

\begin{lemma}[{\cite[Proposition 6.6]{Molino}}]\label{lemma:exponentialmap}
Let $M$ be a connected manifold with a Riemannian foliation $\F$ and a transversally complete bundle-like metric $g$. Let $P$ be a plaque of a leaf of $\F$. The foliation naturally induced on the normal bundle $\nu\F|_{P}$ restricted to $P$ is denoted by $\mathcal{G}_{P}$. Then, the following properties hold:
\begin{itemize}
\item[(i)] The exponential map $\exp\colon(\nu\F|_{P},\mathcal{G}_{P})\longrightarrow(M,\F)$ is a well defined foliated map.
\item[(ii)] If $P$ is relatively compact, then there exists an open neighborhood $U$ of the zero section of $\nu\F|_{P}$ such that $\exp|_{U}$ is a diffeomorphism.
\end{itemize}
\end{lemma}

\begin{proof}[Proof of Theorem~\ref{Theorem: dichotomy}]
Let $L$ be a non-compact and non-proper leaf of $\F$. We will prove that $\overline{L}$ is compact. Let $T$ be a transversal of $\F$ containing a  point $x\in\overline{L}$. Consider  the metric on $T$ induced by the transverse metric on $(M,\F)$. Let $K$ be an open disk in $T$ centered at $x$ whose  radius is small enough so that $\overline{K}$ is compact. We identify $L$ with $\RR$. Since $x\in \mathop{\rm int}(K)\cap\overline{L}$ and $L$ is non-proper, we can assume that there exists a strictly monotonically increasing sequence $\{x_{i}\}_{i \in \ZZ_{\geq 0}}$ in $L \cap K$ such that $\lim_{i \to \infty} x_{i} = \infty$.

First, by reductio ad absurdum, we will show that there exists a strictly monotonically increasing sequence $\{x_{i}\}_{i \in \ZZ_{<0}}$ in $L \cap K$ such that $\lim_{i \to -\infty} x_{i} = -\infty$. Assume that such $\{x_{i}\}_{i \in \ZZ_{<0}}$ does not exist. Then there exists $z$ in $L \cap K$ such that $K \cap \{ y \in L \mid y < z \}=\emptyset$. For $r > 0$, let $D(r)$ be the disk in $T$ centered at $z$ of radius $r$. Take a small $r$ so that $D(r)\subset K$. Without losing generality we can assume that $D(r)$ is orthogonal to $\F$ and that $x_{0}$, $x_{1}$ and $x_{2}$ belong to $D(r/2)$. Clearly, at each point $y$ in $D(r/2)$, the image of the disk of radius $r/2$ centered at $0$ in $\nu_{y}\F$ under the exponential map at $y$ is contained in $D(r)$. Let $J$ be the closed segment in $L$ which connects $x_{0}$ and $x_{2}$. Then, by Lemma~\ref{lemma:exponentialmap}-(i), a leaf of $\nu \F|_{J}$ is mapped onto a closed neighborhood $S$ of $z$ in $L$ by the normal exponential map along $J$. By construction, $S$ intersects $D(r)$ at its endpoints. Thus it contradicts with the hypothesis $K \cap \{ y \in L \mid y < z \}=\emptyset$. Thus, there exists a strictly monotonically increasing sequence $\{x_{i}\}_{i \in \ZZ}$ in $L \cap K$ such that
\[
\lim_{i \to -\infty} x_{i} = -\infty\;, \quad \lim_{i \to \infty} x_{i} = \infty\;.
\]
Since $\overline{K}$ is compact, both $\{x_{i}\}_{i \in \ZZ_{\geq 0}}$ and $\{x_{-i}\}_{i \in \ZZ_{\geq 0}}$ have an accumulation point in $\overline{K}$. By taking subsequences, we can assume that both $\{x_{i}\}_{i \in \ZZ_{\geq 0}}$ and $\{x_{-i}\}_{i \in \ZZ_{\geq 0}}$ are Cauchy sequences in $\overline{K}$.
Let $J_{i}$ be the segment in $L$ which connects $x_{i}$ and $x_{i+1}$. We fix $r_{0} > 0$. By the transversally completeness of $\F$ and Lemma~\ref{lemma:exponentialmap}-(ii), we can take a subset $U_{i}$ of $M$ by $U_{i} = \cup_{y \in J_{i}} \exp_{y}(A_{y})$, where $\exp_{y} : \nu_{y} \F \to M$ is the normal exponential map at $y$ and $A_{y}$ is the disk of radius $r_{0}$ centered at $0$ in $\nu_{y} \F$ (note that $\exp_{y}|_{A_{y}}$ may not be a diffeomorphism, but it is not important here). Then $\{U_{i}\}_{i\in \ZZ}$ covers $\overline{L}\cap \overline{K}$. Since $\overline{L}\cap \overline{K}$ is compact, we choose a finite subset $I\subset \ZZ$ so that $\{U_{i}\}_{i\in I}$ covers $\overline{L}\cap \overline{K}$. Then, by construction, $\overline{L}$ is contained in $\cup_{i \in I} U_{i}$, which is a relatively compact subset of $M$. So $\overline{L}$ is compact.

We will show that if $\F$ admits a leaf $L$ whose closure is compact, then so is the closure of any leaf $L'$ of $\F$. Let $d$ denote the metric on $M$ induced by $g$. Here $l=d(\overline{L},\overline{L'})$ is bounded by the connectivity of $M$. The transversally completeness implies that, for any point $y\in \overline{L'}$, there exists a geodesic which connects $y$ and $\overline{L}$ of length $l$ (see~\cite[Lemma~4]{Reinhart} for the transport of orthogonal geodesics along leaves). Thus $\overline{L'}$ is contained in the $l$-neighborhood of $\overline{L}$, which implies compactness of $\overline{L'}$.

To complete the dichotomy, assume that $\F$ admits a non-compact proper leaf. Then, by the preceding arguments, every other leaf must be non-compact and proper. Moreover, by a theorem of \u{Z}ukova~\cite[Theorem~1]{Zukova}, any non-compact proper leaf $L$ of $\F$ admits an open tubular neighborhood $U$ such that $(U,\F|_{U})$ is an $\RR$-bundle. Thus, $\F$ is an $\RR$-bundle.
\end{proof}

A trivial $\RR$-bundle is taut with a product metric constructed with a trivialization. Any $\RR$-bundle admits a flat connection whose holonomy group is $\{\pm 1\}$ whose double cover is a trivial $\RR$-bundle. Such an $\RR$-bundle is taut with a metric whose lift to the double cover is a product metric. Thus, any $\RR$-bundle is taut. We get the following corollary:
\begin{corollary}
If $\F$ admits a leaf whose closure is non-compact, then $\F$ is taut.
\end{corollary}
Below we will consider the case where the closure of every leaf is compact.

\section{Linearly foliated torus bundles}\label{sec:4}

\subsection{Definition of $(\mathbb{T}^k,\mathcal{F}_v)$-bundles}
Let $v$ be a nonzero vector in $\mathbb{R}^{k}$. Let $\mathcal{F}_{v}$ be the linear flow of slope $v$ on the torus $\mathbb{T}^{k}$.
\begin{definition}
A linearly foliated torus bundle, or a $(\mathbb{T}^k,\mathcal{F}_v)$-bundle is a torus bundle equipped with a defining $1$-cocycle valued in $\Diff(\mathbb{T}^{k},\mathcal{F}_{v})$.
\end{definition}

We will always assume that the leaves of $\mathcal{F}_{v}$ are dense in $\mathbb{T}^{k}$. Notice that the total space of a $(\mathbb{T}^{k},\mathcal{F}_{v})$-bundle has a one dimensional foliation, which is Riemannian. This foliation is called the {\em canonical foliation} of the $(\mathbb{T}^k,\mathcal{F}_v)$-bundle.

\subsection{Reduction to linearly foliated torus bundles}

The importance of the linearly foliated torus bundles comes from the following version of Molino structure theory for Riemannian foliations such that the closures of leaves are compact.

\begin{theorem}\label{Theorem: Molino}
Let $(M,\mathcal{F})$ be a manifold with a Riemannian foliation of dimension one and codimension $q$ such that the closure of each leaf is compact. Let $p \colon M^{1} \longrightarrow M$ be the orthonormal frame bundle of $\nu \F$. We have:
\begin{enumerate}
\item There exists an $\OO(q)$-invariant one dimensional Riemannian foliation $\mathcal{F}^{1}$ on $M^{1}$ such that the restriction of $p$ to each leaf of $\mathcal{F}^{1}$ is a covering map to a leaf of $\mathcal{F}$.
\item There exists a smooth $\OO(q)$-equivariant $(\mathbb{T}^k,\mathcal{F}_v)$-bundle $\pi_{b} \colon M^1 \longrightarrow W$ whose fibers are the closures of the leaves of $\mathcal{F}^1$.
\end{enumerate}
\end{theorem}

\begin{proof}[Outline of the proof]
We refer to~\cite[Chapters~4 and~5]{Molino} or~\cite[Chapter~4]{Moerdijk Mrcun} for the terminologies on the Molino theory. The part (i) is a general construction valid for any Riemannian foliation. We give an outline of the proof of the part (ii). $(M^{1},\F^{1})$ has a transverse parallelism $\{X_{1}, \ldots, X_{m}\}$ given by the basic connection on $\nu \F$ and the canonical $1$-form. Since the closure of each leaf of $\F$ is compact, each $X_{i}$ is complete on an open neighborhood of the closure of each leaf of $\F^{1}$. This implies that $(M^{1},\F^{1})$ is homogeneous. In particular, the closures of leaves of $\F^1$ define a foliation $\overline{\F^{1}}$ of $M^{1}$ whose leaf space $W=M^{1}/\overline{\F^{1}}$ is a smooth manifold. The restriction of $\F^1$ to a fiber $F$ of $M^{1} \to W$ is a Lie foliation of dimension one. Thus, a theorem of Caron-Carri\`{e}re~\cite{Carriere Caron} implies that $(F,\F^{1})$ is diffeomorphic to $(\mathbb{T}^{k},\F_{v})$.
\end{proof}

\begin{remark}\label{remark:proof of 1.2}
Theorem~\ref{Theorem: Molino} and the following lemma reduce the proof of Theorem~\ref{Theorem : Strong tenseness} to the case of linearly foliated torus bundles with a compact Lie group action.
\end{remark}
Let $(M^{\sharp},\F^{\sharp})$ be $(M^{1},\F^{1})$ if $\F^{1}$ is oriented, and otherwise a double cover of $(M^{1},\F^{1})$ such that $\F^{\sharp}$ is oriented. Let $G$ be $\OO(q)$ if $\F^{1}$ is oriented, and otherwise $\OO(q) \ltimes \ZZ/2\ZZ$.

\begin{lemma}\label{lema:reduction}
If $(M^{\sharp},\F^{\sharp})$ admits a $G$-invariant strongly tense metric $g^{\sharp}$ such that the $G$-orbits are orthogonal to $\F^{\sharp}$, then $(M,\F)$ admits a strongly tense metric.
\end{lemma}

\begin{proof}
We have a $g^{\sharp}$-orthogonal decomposition $TM^{\sharp}=\ker\pi_*\oplus T\F^{\sharp}\oplus D$, where $D=(T\F^{\sharp})^{\perp g^{\sharp}}\cap(\ker\pi_*)^{\perp g^{\sharp}}$. As $g^{\sharp}$ is $G$-invariant and $\pi$ is a principal $G$-bundle, there exists a Riemannian metric $g$ on $M$ such that $g^{\sharp}(v,w)=g(\pi_*v,\pi_*w)$ for every $v,w\in T_x\F^{\sharp}\oplus D_x$. Notice that
\begin{equation}\label{formula:bundle decomposition}
\pi_*(T\F^{\sharp})=T\F\quad\text{and}\quad \pi_*D=(T\F)^{\perp g}\;.
\end{equation}
We now show that $g$ is strongly tense. Let $U^{\sharp}$ be an open set of $M^{\sharp}$ and $U=\pi(U^{\sharp})$. We assume that $U^{\sharp}$ and $U$ are simply connected, and hence $\F|_{U^{\sharp}}$ and $\F|_U$ are orientable. We fix compatible orientations of $\F|_{U^{\sharp}}$ and $\F|_U$. Let $\chi_{U^{\sharp}}$ and $\chi_U$ be characteristic forms. Then~\eqref{formula:def chii} and~\eqref{formula:bundle decomposition} imply that $\pi^*\chi_U=\chi_{U^{\sharp}}$. By Rummler's formula~\eqref{formula:rummler}, we get $\pi^*(\kappa|_U)=\kappa^{\sharp}|_{U^{\sharp}}$, and thus the proof is concluded.
\end{proof}

\subsection{Retracting the structure groups of $(\mathbb{T}^{k},\mathcal{F}_{v})$-bundles}\label{section:reduction}

We consider the following group:
\begin{equation*}
\GL_{v}(k;\mathbb{Z}) = \{ A \in \GL(k;\mathbb{Z}) \mid Av = \lambda v, \exists \lambda \in \mathbb{R} \}.
\end{equation*}
We have a standard injection
\begin{equation*}
\iota \colon \GL_{v}(k;\mathbb{Z})\ltimes \mathbb{T}^{k} \longrightarrow \Diff(\mathbb{T}^{k},\mathcal{F}_{v})
\end{equation*}
where $\iota(A,y)(x)=Ax + y$ where $+$ is the sum of $\mathbb{T}^k\equiv\mathbb{R}^k/\mathbb{Z}^k$. We denote the linear part $\GL_{v}(k;\mathbb{Z}) \longrightarrow \Diff(\mathbb{T}^{k},\mathcal{F}_{v})$ of $\iota$ by $\iota_{1}$.

\begin{lemma}\label{Lemma : pi 0}
The map $\pi_{0}(\iota_{1}) \colon \GL_{v}(k;\mathbb{Z}) \to \pi_{0}(\Diff(\mathbb{T}^{k},\mathcal{F}_{v}))$ induced by $\iota_{1}$ is bijective.
\end{lemma}

\begin{proof}
Consider the natural map
\[
\rho_{0} \colon \pi_{0}(\Diff(\mathbb{T}^{k},\mathcal{F}_{v})) \to \Aut(H_1(\mathbb{T}^{k};\mathbb{Z})) \cong \GL(k;\mathbb{Z})\;.
\]
Since $\rho_{0} \circ \pi_{0}(\iota_{1})$ is injective as shown in the proof of~\cite[Lemma~III.2]{Molino Sergiescu}, $\pi_{0}(\iota_{1})$ is injective. We will show that $\pi_{0}(\iota_{1})$ is surjective. It suffices to show that the image of $\rho_{0}$ is contained in $\GL_{v}(k;\mathbb{Z})$. For every $f\in\pi_{0}(\Diff(\mathbb{T}^{k},\mathcal{F}_{v}))$, we consider the following commutative diagram
\begin{equation*}
\xymatrix{ H_{1}(\mathbb{T}^{k};\mathbb{R}) \ar[r]^<<<<<{\psi} \ar[d]_{\rho_{0}(f)} & H^{1}(\mathbb{T}^{k}/\F_{v})^{*} \ar[d]^{\rho_{0}(f)^{**}} \\
 H_{1}(\mathbb{T}^{k};\mathbb{R}) \ar[r]_>>>>>{\psi} & H^{1}(\mathbb{T}^{k}/\F_{v})^{*}\;,}
\end{equation*}
where $\psi$ is the map induced from the canonical pairing $H_{1}(\mathbb{T}^{k};\mathbb{R}) \times H^{1}(\mathbb{T}^{k}/\F_{v}) \to \RR$. Since the kernel of $\psi$ is generated by $v$ and both vertical arrows are isomorphisms, it follows that $v$ is an eigenvector of $\rho_{0}(f)$.
\end{proof}

\begin{proposition}\label{Proposition : Retraction 1}
$\Diff(\mathbb{T}^{k},\mathcal{F}_{v})$ retracts to $\GL_{v}(k;\mathbb{Z})\ltimes \mathbb{T}^{k}$.
\end{proposition}

\begin{proof}
By Lemma~\ref{Lemma : pi 0}, the injection $\GL_{v}(k;\mathbb{Z})\ltimes \mathbb{T}^{k} \longrightarrow \Diff(\mathbb{T}^{k},\mathcal{F}_{v})$ induces the bijection on the groups of connected components. Then it suffices to show that the identity component $\Diff_{0}(\mathbb{T}^{k},\mathcal{F}_{v})$ retracts to the identity component $\mathbb{T}^{k}$ of $\GL_{v}(k;\mathbb{Z})\ltimes \mathbb{T}^{k}$. Since $\Diff_{0}(\mathbb{T}^{k},\mathcal{F}_{v})$ acts trivially on the homology of $\mathbb{T}^{k}$, this is a direct consequence of~\cite[Lemmas~II.2 and~III.2]{Molino Sergiescu}.
\end{proof}

\subsection{Tenseness of linearly foliated torus bundles}\label{sec:44}

Let $\pi : M \to W$ be an oriented $(\mathbb{T}^{k},\F_{v})$-bundle with canonical foliation $\F$. Assume that the structure group of $\pi$ is reduced to $\GL_{v}(k;\ZZ) \ltimes \mathbb{T}^{k}$, namely, $\pi$ is associated to a $1$-cocycle $\sigma$ on $W$ valued in $\GL_{v}(k;\ZZ) \ltimes \mathbb{T}^{k}$. By the discreteness of $\GL_{v}(k;\ZZ)$, the $1$-cocycle $\sigma$ yields a homomorphism $h_{\sigma} : \pi_{1}M \to \GL_{v}(k;\ZZ)$. Here, clearly $h_{\sigma}$ is trivial if and only if $\sigma$ yields a structure of a principal $\mathbb{T}^{k}$-bundle on $\pi$. Thus, restricting $\sigma$ to a simply connected open set $U$ of $W$, we get a principal $\mathbb{T}^{k}$-action $\rho_{U}$ on $\pi^{-1}(U)$. We denote $\rho_{U}$ by $\rho_{U}(\sigma)$, since $\rho_{U}$ is determined only by $\sigma$ up to the change of coordinates on $\mathbb{T}^{k}$ by $\GL_{v}(k;\ZZ)$.

The following holds for an oriented $(\mathbb{T}^{k},\F_{v})$-bundle $\pi : M \to W$ with canonical foliation $\F$.
\begin{lemma}\label{lem:char}
A characteristic form $\chi$ of $(M,\F)$ is strongly tense if and only if there exists a reduction $\sigma$ of the structure group of $\pi$ to $\GL_{v}(k;\ZZ) \ltimes \mathbb{T}^{k}$ such that, for any simply connected subset $U$ of $W$, the characteristic form $\chi|_{\pi^{-1}(U)}$ is $\rho_{U}(\sigma)$-invariant. If $\pi$ has a structure of principal $\mathbb{T}^{k}$-bundle and $\chi$ is invariant under the principal $\TT^{k}$-action, then $\chi$ is taut after a multiplication of a positive function.
\end{lemma}

\begin{proof}
Let $x$ be a point of $W$ and take a simply connected neighborhood $U$ of $x$ in $W$. Let $X$ be a vector field on $\pi^{-1}(U)$ tangent to $\F|_{\pi^{-1}(U)}$ such that the closure of the flow generated by $X$ is equal to a principal $\TT^{k}$-action on $\pi^{-1}(U)$. The ``if'' part follows from the following local computation. By Rummler's formula~\eqref{formula:rummler}, we get
\begin{equation}\label{eq:k}
\kappa = \iota_{\frac{1}{\chi(X)}X}d\chi = \frac{1}{\chi(X)} \iota_{X}d\chi = -\frac{1}{\chi(X)} d(\chi(X)) = -d \log |\chi(X)|\;.
\end{equation}
Since $\log |\chi(X)|$ is a basic function and each $\pi^{-1}(U)$ is saturated, $\kappa$ is basic and closed in $M$.

We show the ``only if'' part. Assume that $\chi$ is strongly tense. Since $\kappa = \iota_{\frac{1}{\chi(X)}X}d\chi$ is basic and closed, there exists a basic function $h$ on $(\pi^{-1}(U),\F|_{\pi^{-1}(U)})$ such that $dh = \kappa$. The Rummler's formula implies that the mean curvature form of a characteristic form $e^{-h}\chi$ is zero. Thus the closure of the flow generated by the vector field $X$ tangent to $\F$ such that $e^{-h}\chi(X) = 1$ is a principal $\mathbb{T}^{k}$-action $\rho_{U}$ which preserves $e^{-h}\chi$. By covering $W$ with simply connected open subsets $\{U_{i}\}$, these $\rho_{U_{i}}$ yield a reduction $\sigma$ of the structure group of $\pi$ into $\GL_{v}(k;\ZZ) \ltimes \mathbb{T}^{k}$ such that $\rho_{U_{i}}(\sigma) = \rho_{U_{i}}$. Since $e^{-h}$ is basic, $\rho_{U_{i}}(\sigma)$ preserves $\chi$.

The latter part of the statement follows, because~\eqref{eq:k} implies that the mean curvature form of $e^{-\chi(X)} \chi$ is zero.
\end{proof}

\subsection{Molino's commuting sheaf of linearly foliated torus bundles}\label{sec:mol}

Let $\pi : M \to W$ be a $(\mathbb{T}^{k},\F_{v})$-bundle with canonical foliation $\F$ whose structure group is a subgroup of $\GL_{v}(k;\ZZ) \ltimes \mathbb{T}^{k}$. The Molino's commuting sheaf $\mathcal{C}$ of $(M,\F)$ is determined by the structure group as follows (we refer to~\cite[Section 5.3]{Molino} for the definition of the Molino's commuting sheaf). The following proposition is a direct consequence of {\cite[Proposition~6]{Nozawa}}:

\begin{proposition}\label{prop:hol}
The holonomy homomorphism $\hol(\mathcal{C})$ of $\mathcal{C}$ is determined by
\begin{equation*}
\xymatrix{ \pi_{1}M \ar[r]  & \GL_{v}(k;\mathbb{Z}) \ar[r]^<<<<<{r} & \GL(k-1;\RR)\;,}
\end{equation*}
where the first arrow is the holonomy homomorphism $h_{\sigma}$ of $\pi$ described in the first paragraph of the last section and the second arrow $r$ is defined by sending $A\in \GL_{v}(k;\mathbb{Z})$ to the map $\RR^{k}/\RR v \to \RR^{k}/\RR v$ induced by $A$.
\end{proposition}

We will prove Theorem~\ref{thm:alv} by using Proposition~\ref{prop:hol}.
\begin{proof}[Proof of Theorem~\ref{thm:alv}]
By Theorem~\ref{Theorem: dichotomy}, it suffices to prove the theorem in the case where the closure of each leaf is compact. Let $M^{1} \to W$ be the oriented linearly foliated torus bundle in Theorem~\ref{Theorem: Molino}. If we have a strongly tense metric on $(M,\F)$ with mean curvature form $\kappa$, then we can construct a strongly tense characteristic form on $(M^{1},\F^{1})$ with mean curvature form $\kappa^{1}$ such that $\kappa^{1} = \pi^{*}\kappa$. Thus it suffices to prove the theorem in the case where $(M,\F)$ is an oriented linearly foliated torus bundle. Let $\gamma : \mathbb{S}^{1} \to W$ be a loop in $W$. The pull-back of the oriented linearly foliated torus bundle to $\mathbb{S}^{1}$ is a mapping torus $N=\mathbb{T}^{k} \times [0,1]/(A(x),0) \sim (x,1)$ for some $A \in \GL_{v}(k;\ZZ)$. By Lemma~\ref{lem:char}, the restriction of $\chi$ to $\mathbb{T}^{k} \times \{t\}$ is linear with respect to the standard coordinate on $\TT^{k}$. By~\cite[Example~7.3]{Nozawa}, the class $[\kappa]$ is determined by $[\kappa]|_{\mathbb{T}^{k} \times \{t\}}=0$ and $\int_{\mathbb{S}^{1}} \kappa = \log \lambda$, where $\lambda$ is the eigenvalue of $A$ with respect to $v$. Thus $[\kappa]$ is determined by $(M,\F)$. The latter part follows from Proposition~\ref{prop:hol}.
\end{proof}

Let $\det \mathcal{C}$ be the determinant line bundle of $\mathcal{C}$. Since we are assuming that the linear flow $\F_{v}$ tangent to $v$ is dense in $\mathbb{T}^{k}$, there exists no non-trivial $\ZZ$-linear relation on the entries of $v$. It follows that the composite of
\[
\xymatrix{ \GL_{v}(k;\mathbb{Z}) \ar[r]^<<<<<{r} & \GL(k-1;\RR) \ar[r]^<<<<<{\det} & \GL(1;\RR)\;,}
\]
is injective. We get the following consequence of Proposition~\ref{prop:hol}, which is necessary in the proof of Corollary~\ref{corollary:tautness characterization}.
\begin{corollary}\label{cor:C}
The image of the holonomy homomorphism of $\det \mathcal{C}$ is contained in $\{\pm 1 \}$ if and only if so is the image of the holonomy homomorphism of $\mathcal{C}$.
\end{corollary}

\begin{proof}[Proof of Corollary~\ref{corollary:tautness characterization}]
The equivalence of the first three assertions is a formal consequence of Theorem~\ref{corollary:tautness characterization} and Corollary~\ref{corollary:duality} (see the proof of~\cite[Theorems~3.3 and~3.5]{royo-saralegi-wolak2}). We will show the equivalence of (iii) and (iv). If $(M,\F)$ is an $\RR$-bundle, then both (iii) and (iv) are true. By Theorem~\ref{Theorem: dichotomy}, we can assume that the closure of each leaf of $(M,\F)$ is compact. The equivalence of (iii) and (iv) follows from ~\eqref{eq:tpdS}. Finally, since the holonomy homomorphisms of Sergiescu's orientation sheaf $\mathcal{P}$ and the determinant line bundle of Molino's commuting sheaf are equal up to sign by definition of $\mathcal{P}$, the equivalence of (iv) and (v) follows from Corollary~\ref{cor:C}.
\end{proof}

The Molino's commuting sheaf and the \'{A}lvarez class are illustrated with Carri\`{e}re's example~\cite{Carriere}.
\begin{example}
Take $A\in\SL(2;\mathbb{Z})$ such that $\tr A > 2$, let $\lambda$ be one of its eigenvalues and denote by $v=(a,b)\in\RR^2$ the corresponding eigenvector. Notice that $A$ induces a diffeomorphism $\overline{A}$ on $\TT^2=\RR^2/\ZZ^2$. Consider the manifold $\TT^3_A=(\TT^2\times[0,1])/(\overline{A}x,0)\sim(x,1)$, which is a $\TT^{2}$-bundle over $\mathbb{S}^1$. Here $\TT^3_A$ admits the structure of a $(\TT^{2},\F_{v})$-bundle whose structure group is the infinite cyclic subgroup of $\SL_v(2;\mathbb{Z})$ generated by $A$. Let $\F$ be the canonical foliation. Taking the standard coordinates $(x,y,t)$ on $\TT^3_A$, we have the parallelism
\begin{equation*}
X= \lambda^{t}(a\partial/\partial{x} + b\partial/\partial{y})\;,\quad
Y=\lambda^{-t}(-b\partial/\partial{x} + a\partial/\partial{y})\;,\quad
T=\partial/\partial{t}\;.
\end{equation*}
Let $g$ be the Riemannian metric on $\TT^3_A$ such that $\{X,Y,T\}$ is an orthonormal parallelism. It is straightforward to check that $g$ is bundle-like and that its mean curvature form $\kappa$ is given by $(\log\lambda )dt$ in the standard coordinates. So, $g$ is strongly tense, while the \'Alvarez class $[\kappa]$ of $\F$ is not trivial. Thus $\F$ is not taut.
\end{example}

\section{Invariant tense metrics on linearly foliated torus bundles}\label{sec:5}

Let $\pi : M \to W$ be a $(\mathbb{T}^{k}, \mathcal{F}_{v})$-bundle and $G$ a compact Lie group acting on $M$ preserving the canonical foliation $\mathcal{F}$. We assume that $\F$ is oriented. In this section, we will prove the following result, which completes the proof of Theorem~\ref{Theorem : Strong tenseness} (see Remark~\ref{remark:proof of 1.2}):
\begin{theorem}\label{thm:tense2}
$(M,\F)$ admits a $G$-invariant strongly tense metric. Moreover, if the $G$-action is locally free and $\F$ is not tangent to any $G$-orbit, then we can take a $G$-invariant strongly tense metric so that the $G$-orbits are orthogonal to $\F$.
\end{theorem}

By Proposition~\ref{Proposition : Retraction 1}, we can assume that the structure group of $\pi$ is $\GL_{v}(k;\mathbb{Z})\ltimes \mathbb{T}^{k}$. Let $\phi_{M}$ be the composite of
\begin{equation}\label{eq:compM}
\xymatrix{ \pi_{1}M \ar[r] & \GL_{v}(k;\mathbb{Z}) \ar[r] & \mathbb{R}\;, }
\end{equation}
where the first arrow is  the holonomy homomorphism $h_{\sigma}$ of $\pi$ described in the first paragraph of Section~\ref{sec:44} and the second arrow maps $A\in \GL_{v}(k;\mathbb{Z})$ to the logarithm of the eigenvalue of $A$ with respect to $v$.

\begin{lemma}\label{Lemma: holonomy}
Let $K$ be a $G$-orbit in $M$ and $i_{K} : \pi_{1}K \to \pi_{1}M$ be the map induced by the inclusion. Then, $\phi_{M} \circ i_{K}$ is trivial.
\end{lemma}

\begin{proof}
Let $E$ be the vector subbundle of $\nu \F$ defined by the kernel of $\pi_{*} : \nu \F \to TW$. Here $E$ is invariant under the $G$-action, because $G$ preserves $\F$ and the fibers of $\pi$ are closures of the leaves of $\F$. Then $E$ has a $G$-invariant metric by compactness of $G$. Thus, for any loop $\gamma$ in $K$, the holonomy map associated to $\pi_{*}\gamma$ preserves a metric on $E$. This implies the triviality of $\phi_{M} \circ i_{K}$.
\end{proof}

The key of our proof of Theorem~\ref{thm:tense2} is the following lemma, which was already used in~\cite[Lemma~6.3]{Alvarez Lopez}.
\begin{lemma}\label{lemma:chi1}
Let $\chi$ be a characteristic form of $(M,\F)$. Let $\chi_{1}$ be defined by
$$
\chi_{1} = \int_{g \in G} \epsilon(g) (g^{*} \chi) dg\;,
$$
where $dg$ is a Haar measure of $G$ and $\epsilon(g) = 1$ if $g$ preserves the orientation of $\F$ and $\epsilon(g) = -1$ otherwise. Then $\chi_{1}$ is a characteristic form of $\F$. Moreover, if $\chi$ is taut, then so is $\chi_{1}$.
\end{lemma}

\begin{proof}
It is easy to see that $\chi_1$ is positive on $T\F$ and $\ker\chi=\ker\chi_1$, which implies the first part. The latter part follows from Rummler's formula~\eqref{formula:rummler} and $d\chi_{1} = \int_{g \in G} \epsilon(g) (g^{*} d\chi) dg$.
\end{proof}

\begin{remark}
As we saw in the last lemma, the sum of two taut characteristic forms is taut, while the sum of two tense characteristic forms may not be tense. We can use this phenomenon to show Theorem~\ref{thm:tense2} by taking a covering of $(M,\F)$ which is taut. Strongly tenseness of characteristic forms is not linear in a direct way. But there is a certain way to make the sum of two strongly tense characteristic forms to obtain a strongly tense one by using the interpretation of strongly tenseness as a twisted version of tautness (see~\cite[Proposition~7.8]{Nozawa 2}).
\end{remark}

\begin{proof}[Proof of Theorem~\ref{thm:tense2}]
By Proposition~\ref{Proposition : Retraction 1}, the structure group of $\pi$ can be reduced to $\Gamma \ltimes \mathbb{T}^k$, where $\Gamma$ is a subgroup of $\GL_{v}(k;\ZZ)$. Let $p : (M',\mathcal{F}') \to (M,\mathcal{F})$ be the covering of $(M,\mathcal{F})$ such that $\pi_{1}M' \cong \ker \phi_M$, whose covering group is identified with $\Gamma$. For any $G$-orbit $K$ in $M$, we have $\pi_{1}K\subset \pi_{1}M'$ by Lemma~\ref{Lemma: holonomy}. Thus the $G$-action on $M$ lifts to a $G$-action $\psi_{G}$ on $M'$. The $\Gamma$-action on $M'$ via deck transformations commutes with $\psi_{G}$.

By construction, $M'$ has a structure of a principal $\mathbb{T}^{k}$-bundle such that $\F'$ is the orbit foliation of a dense $\RR$-subaction of the principal $\mathbb{T}^{k}$-action $\rho_{0}$. Let $X$ be a vector field which generates the dense $\RR$-subaction of $\rho_{0}$. Let $\chi$ be a $\rho_{0}$-invariant characteristic form of $(M',\F')$ such that $\chi(X)=1$. The latter part of Lemma~\ref{lem:char} implies that $\chi$ is taut. Let $\chi_{1}$ be the characteristic form of $(M',\F')$ obtained from $\chi$ like in Lemma~\ref{lemma:chi1}. Let $X_{1}$ be the vector field tangent to $\F'$ such that $\chi_{1}(X_{1})=1$. Since $\chi$ is taut, by Lemma~\ref{lemma:chi1}, so is $\chi_{1}$. Thus $X_{1}$ is a Killing vector field with respect to a Riemannian metric on $M'$. Then the closure of the flow generated by $X_{1}$ is a principal $\mathbb{T}^k$-action $\rho_{1}$ on $M'$. Since $X_{1}$ is $G$-invariant up to sign determined by $\epsilon$ in Lemma~\ref{lemma:chi1}, $\psi_{G}$ and $\rho_{1}$ yield a $(G \ltimes \mathbb{T}^{k})$-action on $M'$, where the semidirect product is defined by $\epsilon : G \to \{\pm 1\}$.
Here $\rho_{1}$ and the $\Gamma$-action on $M'$ yield a $(\Gamma \ltimes \mathbb{T}^k)$-action on $M'$, because the commutativity of the $\Gamma$-action with $\psi_{G}$ implies, for $h\in \Gamma$,
$$
h_{*}X_{1} = \int_{g \in G} (h_{*}g_{*} X) dg = \int_{g \in G} (g_{*} h_{*}X) dg = \phi_M(h) \int_{g \in G} (g_{*} X) dg= \phi_M(h)X_1\;.
$$
In total, we get a $\big((G \times \Gamma) \ltimes \mathbb{T}^k \big)$-action on $M'$ such that $\F'$ is the orbit foliation of a dense $\RR$-subaction of the principal $\mathbb{T}^{k}$-action $\rho_{1}$.

Let $\chi_{2}$ be a $\Gamma$-invariant characteristic form of $(M',\F')$. Using the $\big((G \times \Gamma) \ltimes \mathbb{T}^k \big)$-action on $M'$ obtained in the last paragraph, let
$$
\chi_{3} = \int_{u \in G \ltimes \mathbb{T}^k} \epsilon'(u) (u^{*} \chi_{2}) du\;,
$$
where $du$ is a Haar measure of $G \ltimes \mathbb{T}^k$ and $\epsilon'(u) = 1$ if $u$ preserves the orientation of $\F'$ and $\epsilon'(u) = -1$ otherwise. Then $\chi_{3}$ is a $(\Gamma \ltimes \mathbb{T}^k)$-invariant characteristic form, which is strongly tense by Lemma~\ref{lem:char}. Thus, $(M',\F')$ admits a $(G \times \Gamma)$-invariant strongly tense metric, which induces a $G$-invariant strongly tense metric on $(M,\F)$.

We show the latter part. Assume that $\F$ is nowhere tangent to $G$-orbits. Then we can take a characteristic form $\chi_{2}$ on $(M,\F)$ so that every $G$-orbit is tangent to $\ker \chi_{2}$. We conclude the proof of the latter part by constructing $\chi_{3}$ from this $\chi_{2}$ like in the last paragraph.
\end{proof}

\end{document}